\documentclass{amsart}
\usepackage{amssymb,amsmath,latexsym}
\usepackage{amsthm}
\usepackage{fontenc}
\usepackage{amssymb}
\numberwithin{equation}{section}

\newtheorem{theorem}{Theorem}[section]
\newtheorem{corollary}{Corollary}[theorem]
\newtheorem{lemma}[theorem]{Lemma}

\begin{document}
\author{Alexander E. Patkowski}
\title{A note on Some Partitions Related to Ternary Quadratic Forms}

\maketitle
\begin{abstract} We offer some partition functions related to ternary quadratic forms, and note on their upper bounds and related properties. We offer these results as an application of a simple method related to conjugate Bailey pairs presented in a prior study, further illustrating its utility. \end{abstract}

\keywords{\it Keywords: \rm ternary quadratic forms; partitions; $q$-series.}

\subjclass{ \it 2010 Mathematics Subject Classification 05A17, 11E20}

\section{Introduction}
We write a positive ternary quadratic form as
\begin{equation}Q(x_1,x_2,x_3)=ax_1^2+bx_2^2+cx_3^2+rx_2x_3+sx_3x_1+tx_1x_2,\end{equation}
and say it is primitive if $gcd(a,b,c,r,s,t)=1.$ We recall some elementary facts about $Q(x,y,z).$ We first consider $(a,b,c,r,s,t)=(1,1,1,0,0,0).$ The number $S(X)$ of integers less than or equal to $X$ which are representable as a sum of three squares is $\sim\frac{5}{6}X.$ (See [8] for an interesting discussion on this result and related material.) Here we used the common notation $F(x) \sim G(x)$ to mean that $F(x)$ is asymptotically equivalent to $G(x),$ or $\lim_{x\to \infty}F(x)/G(x)=1.$  We will also use Vinogradov's $F(x) \ll G(x)$ to mean that $|F(x)|\le \sigma G(x)$ for some constant $\sigma\ge0.$ In general, it is known (e.g.[3, eq.(1)--(2)],[9]), that \begin{equation}\#\{n\le X: n=f(x,y,z) \}\sim \sigma' X,\end{equation} where $\sigma'\in(0,1].$ In the case of binary quadratic forms, Bernay's result (see [2]) says that \begin{equation}\#\{n\le X: n=ax^2+by^2+cxy \}\sim \lambda\frac{X}{\sqrt{\log X}},\end{equation}
for a positive constant $\lambda.$ 
\par In [6] we find $\ll$ results for partition functions related to (1.3), offering some interesting applications of $q$-series arising from Bailey pairs associated with indefinite binary quadratic forms. A simple method used in [7], which uses a definition of a Bailey pair in conjunction with conjugate Bailey pairs, produced the tools needed to obtain $q$-series related to ternary quadratic forms. The motivation of this paper is to obtain partition functions which are related to ternary forms, thereby extracting related properties.
\par Define for $q=e^{2\pi iz}$ the theta function [4, pg.245]
$$T(z):=\sum_{(x_1,x_2,x_3)\in\mathbb{Z}^3}q^{Q(x_1,x_2,x_3)},$$ representing a holomorphic function on the half-plane $\mathfrak{H}:=\{z\in\mathbb{C}: \Im(z)>0\}.$
Recall that $T(z)$ is a modular form of weight $\frac{3}{2}$ on some group $\Gamma_0(N),$ that has the decomposition [4, pg.260--261]
\begin{equation}T(z)=E(z)+C(z),\end{equation}
where $E(z)$ is an Eisenstein series [4, pg.240], and $C(z)$ a cusp form. A cusp form is a modular form wherein its initial coefficient in its Fourier series expansion is $0.$ It is also well-known [4, Theorem 4(i), eq.(13), $k=\frac{3}{2}$] in the ternary case, that the Fourier coefficients of $E(z)$ are $\ll n^{1/2},$ and the Fourier coefficients of $C(z)$ are $\ll n^{3/4}.$ Hence the triangle inequality implies the coefficients of $T(z)$ are $\ll n^{3/4}.$ This observation in relation to our weighted ternary expansions allow us to make related statements on the partition functions presented in Theorem 3.2.

\section{Bailey's Lemma and identities}
Here we discuss the analytic tools used to obtain our generating functions. Put [5] $(Z;q)_n=(Z)_n=(1-Z)(1-Zq)\cdots(1-Zq^{n-1}).$ Bailey [1] introduced the idea of a pair of sequences $(\alpha_n, \beta_n),$ relative to $a,$ which satisfy
\begin{equation}\beta_n(a,q)=\beta_n=\sum_{r\ge0}\frac{\alpha_r(a,q)}{(aq)_{n+r}(q)_{n-r}}.\end{equation}
Further formulas needed for our study are a specialization of Bailey's lemma [1] (relative to $a=q$ and subject to convergence)
\begin{equation}\sum_{n\ge0}(q)_{n}(-1)^n\beta_nq^{n(n+1)/2}=(1-q)\sum_{n\ge0}(-1)^nq^{n(n+1)/2}\alpha_n,\end{equation}
and the pairs [7, eq(3.5)] $(\alpha_n(q^2,q^2), \beta_n(q^2,q^2)),$
\begin{equation}\beta_{n}(q^2,q^2)=\frac{q^{-n}}{(-q)_{2n}},\end{equation}
\begin{equation}\alpha_{n}(q^2,q^2)=(-1)^nq^{n(n-1)}\frac{(1-q^{4n+2})}{1-q^2}\sum_{m\ge0}q^{m(m+1)/2}\sum_{2|j|\le m}(-1)^jq^{-j(j-1)+2nj},\end{equation}
and [7, eq(3.6)] $(\alpha_n, \beta_n),$

\begin{equation}\beta_{n}=\frac{(q)_n(-1)^nq^{n(n-1)/2}}{(q)_{2n}},\end{equation}
\begin{equation}\alpha_{n}=(-1)^nq^{n(n-1)/2}\frac{(1-q^{2n+1})}{1-q}\sum_{m\ge0}(-1)^mq^{m(3m+1)/2}(1-q^{2m+1})\sum_{|j|\le m}(-1)^jq^{-j(j-1)/2+nj}.\end{equation}
Inserting (2.3)--(2.4) and (2.5)--(2.6) into (2.2) gives us the following lemma containing our needed identities, which will be used in the next section.
\begin{lemma} We have for $|q|<1,$
\begin{equation}\sum_{n\ge0}\frac{(q^2;q^2)_n(-1)^nq^{n^2}}{(-q)_{2n}}=\sum_{n\ge0}q^{2n^2}(1-q^{4n+2})\sum_{m\ge0}q^{m(m+1)/2}\sum_{2|j|\le m}(-1)^jq^{-j(j-1)+2nj},\end{equation}
and
\begin{equation}\sum_{n\ge0}\frac{(q)_nq^{n^2}}{(q^{n+1})_{n}}=\sum_{n\ge0}q^{n^2}(1-q^{2n+1})\sum_{m\ge0}(-1)^mq^{m(3m+1)/2}(1-q^{2m+1})\sum_{|j|\le m}(-1)^jq^{-j(j-1)/2+nj}.\end{equation}

\end{lemma}

\section{Partitions} We are concerned with one main generating function from which our partition theorems will follow.
\begin{lemma} Let $A_{k,m}(n)$ be the number of partitions of $n$ where: (i) $k$ appears either once or twice. (ii) All parts $< k$ appear at least twice and at most thrice. (iii) $m$ is the number of parts $> k$ and $\le 2k.$ (iv) All parts are $\le 2k,$ and parts that are $\ge k+1$ and $\le 2k$ may appear any number of times. Further, let $\bar{A}_{k,m}(n)$ be those partitions counted by $A_{k,m}(n)$ with number of parts that are $\le k$ odd minus those with number of parts that are $\le k$ even. Then,
\begin{equation}\sum_{n,m\ge0}\bar{A}_{k,m}(n)a^mq^n=\frac{(q)_kq^{1+1+2+2+\cdots +(k-1)+(k-1)+k}}{(1-aq^{k+1})(1-aq^{k+2})\cdots(1-aq^{2k})}=\frac{(q)_kq^{k(k+1)/2+k(k-1)/2}}{(aq^{k+1})_{k}}.\end{equation}
\end{lemma}
We write out the numerator of the right side of (3.1) for the reader to show the weight associated with parts that are $\le k$ for the sake of clarity.
Define the polynomial in $x$ by $f_k(x):=(x;x)_kx^{1+1+2+2+\cdots +(k-1)+(k-1)+k}.$ Then \begin{equation}f_1(x)=x^{1}-x^{1+1}.\end{equation} \begin{equation}f_2(x)=x^{1+1+2}-x^{1+1+1+2}-x^{1+1+2+2}+x^{1+1+1+2+2}.\end{equation}\begin{equation}f_3(x)=x^{2+2+1+1+3}+x^{1+1+1+2+2+2+3}+x^{1+1+1+2+2+3+3}\end{equation}
$$+x^{1+1+2+2+2+3+3}-x^{1+1+1+2+2+3}-x^{1+1+2+2+2+3}-x^{1+1+2+2+3+3}-x^{1+1+1+2+2+2+3+3}.$$
It is seen from (3.2)--(3.4) that the weight is $+1$ when the number of parts is odd and $-1$ if the number of parts is even.
\par Put 
\begin{equation} \sum_{m,k\ge0}\bar{A}_{k,m}(n)=B(n),\end{equation}
and
\begin{equation} \sum_{m,k\ge0}(-1)^{m+k}\bar{A}_{k,m}(n)=\bar{B}(n).\end{equation}
\begin{theorem} We have,
\begin{equation}\sum_{n\ge0}\frac{(q^2;q^2)_n(-1)^nq^{n^2}}{(-q)_{2n}}=\sum_{n\ge0}\frac{(q)_n(-1)^nq^{n(n+1)/2+n(n-1)/2}}{(-q^{n+1})_{n}}=\sum_{n\ge0}\bar{B}(n)q^n,\end{equation}
\begin{equation}\sum_{n\ge0}\frac{(q)_nq^{n^2}}{(q^{n+1})_{n}}=\sum_{n\ge0}B(n)q^n.\end{equation}
\end{theorem}

As a natural consequence of our partition function having an intimate connection with a ternary quadratic forms, we are able to state the following result.
\begin{theorem} For $n>1,$ if $n$ is not of the form $Q_1(l,m,j):=2l^2+m(m+1)/2-j(j\pm1)+2lj$ for $2|j|\le m, l>0,$ then $\bar{B}(n)=-\bar{B}(n-2)$ and $\bar{B}(n)=0.$ Furthermore, $\bar{B}(n)\ll n^{3/4}$ for all $n\ge0.$ Similarly, for $n>1$ if $n$ is not of the form $Q_2(l,m,j):=l^2+ m(3m+1)/2-j(j\pm1)/2+lj$ for $|j|\le m, l>0,$ or $Q_3(l,m,j):=l^2+m(3m-1)/2-j(j\pm1)/2+lj$ for $|j|< m, l>0,$ then $B(n)=-B(n-1)$ and $B(n)=0.$ Furthermore, $B(n)\ll n^{3/4}$ for all $n\ge0.$
\end{theorem}
\begin{proof} We first remark that in the case of $n=0$ in (2.4) and (2.6), the inner sums are special cases of applying the unit Bailey pair to their corresponding conjugate pair from [7, pg.267, eq.(2.6)--(2.7), $z=1$], giving

\begin{equation}\sum_{m\ge0}q^{m(m+1)/2}\sum_{2|j|\le m}(-1)^jq^{-j(j-1)}=1,\end{equation}
in the case of $\bar{B}(n),$ and
\begin{equation}\sum_{m\ge0}(-1)^mq^{m(3m+1)/2}(1-q^{2m+1})\sum_{|j|\le m}(-1)^jq^{-j(j-1)/2}=1,\end{equation}
in the case of $B(n).$ The first part of the theorem follows from considering the first term $n=0$ of the sum on the right hand side of (2.7), and the second part follows from considering the first term $n=0$ of the sum on the right hand side of (2.8).\par For the estimate, we may observe that the second sum in (2.7) may be written in the form
\begin{equation}\sum_{n\ge0}q^{2n^2}(1-q^{4n+2})\sum_{m\ge0,j\in\mathbb{Z}}(-1)^jq^{m(m+1)/2+j(j+1)+(2m+1)|j|+2jn}.\end{equation}
This can be observed from shifting summation indices or tracing the origins of the proofs of our Bailey pairs from their original conjugate pair. Similarly, the second sum in (2.8) may also be rewritten as a $q$-series related to a positive ternary quadratic form. At this point, we recall our observation from the introduction that 
$$\sum_{\substack{(x_1,x_2,x_3)\in\mathbb{Z}^3\\ Q(x_1,x_2,x_3)=n}}1\ll n^{3/4}.$$
Hence, defining a weight $\omega(x_1,x_2,x_3)$ associated with inequalities on the summation indices $(x_1,x_2,x_3)$ which assumes values from the set $\{0,1,-1\},$ we see that
$$\sum_{\substack{(x_1,x_2,x_3)\in\mathbb{Z}^3\\ Q(x_1,x_2,x_3)=n}}\omega(x_1,x_2,x_3)\ll\sum_{\substack{(x_1,x_2,x_3)\in\mathbb{Z}^3\\ Q(x_1,x_2,x_3)=n}}1\ll n^{3/4}.$$ The result now follows from these observations when equating coefficients of $q^n$ in Lemma 2.1.
\end{proof} 
\begin{corollary} If $n$ is of the form $m(m+1)/2$ and not expressible as $Q_1(l,m,j)$ for $2|j|\le m, j\neq0, l>0,$ then $\bar{B}(n)=1.$ Furthermore, if $n$ is not expressible as $Q_2(l,m,j)$ for $|j|\le m, j\neq0, l>0,$ and is of the form $ m(3m+1)/2$ then $B(n)=(-1)^m,$ and if not expressible as $Q_3(l,m,j)$ for $|j|< m, j\neq0, l>0,$ and is of the form $ m(3m-1)/2$ then $B(n)=(-1)^{m+1}.$
\end{corollary}
\begin{proof} This follows from Theorem 3.3 by considering the case $j=0$ and the case $l=0.$ Having already noted the case $n=0$ in (3.9)--(3.10) we move on to the $j=0$ case. Noting that $\sum_{n\ge0}q^{n^2}(1-q^{2n+1})=1,$ leads one to consider the remaining series in $m$ for (2.7) and (2.8) and the result follows. \end{proof}
We remark about the existence of $n$ of the form $m(m+1)/2$ not expressible as $Q_1(l,m,j)$ for $2|j|\le m, j\neq0, l>0.$ For an example, consider the triangular number $3.$ Since $l>0,$ we are forced to select $l=1,$ giving $2l^2=2,$ as $j>0$ implies the smallest $m$ possible is $2,$ since  $2|j|\le m,$ and the only possibility is $Q_1(0,m,0)=m(m+1)/2=3.$ The only partition for $\bar{B}(3)$ occurs when $k=1$ in Lemma 3.1, there are no parts $<k$ and $1$ can only appear once, because only one $2$ may appear $>k$ and $\le 2k.$ The number of parts $\le k$ is odd, so the weight becomes $(-1)^{m+k}=(-1)^{1+1}=1,$ giving $\bar{B}(3)=1.$ For the triangular number $6,$ we can consider $l=1$ and possibly $l=2.$ For considering $\min_{j,m,2|j|\le m}Q_1(3,m,j),$ we need find the smallest of $2(3)^2+m(m+1)/2-j(j-1)+6j,$ which occurs when $j=-1, m=2,$ giving $13.$ If $l=1$ with $m=3,$ and $j=-1$ then we get $4$ and $6.$ This expression for $6$ occurs from the sum
$$-\sum_{l\ge1}q^{2l^2}\sum_{m\ge0}q^{m(m+1)/2}\sum_{2|j|\le m}(-1)^jq^{-j(j+1)+2lj},$$
which carries a weight of $+1$ when $(l,m,j)=(1,3,-1).$
Likewise $Q_1(1,2,1)=7,$ and also $Q_1(2,2,-1)=5.$ This shows $6$ is not only expressible as $Q_1(0,3,0).$ Consider the triangular number $10.$ Since $l>0,$ we are forced to select $l=1,$ or $l=2$ due to $2l^2.$ If $l=1,$ then $m=2$ with $j=1$ fails, but $m=3$ with $j=1$ gives $Q_1(1,3,1)=2+6+0+2=10.$

1390 Bumps River Rd. \\*
Centerville, MA
02632 \\*
USA \\*
E-mail: alexpatk@hotmail.com, alexepatkowski@gmail.com
\end{document}